\documentclass[a4paper,10pt]{scrartcl}

\usepackage{amsmath, amstext, paralist, amsthm, amssymb, epsfig, paralist}
\usepackage[mathcal]{euscript}

\newfont{\theoremfont}{cmssbx12 scaled 875}
\newtheoremstyle{Eins}{\topsep}{\topsep}{\itshape}{}{\theoremfont}{.}{5pt}{\thmname{#1}\thmnumber{ #2}\thmnote{ #3}}
\newtheoremstyle{Zwei}{\topsep}{\topsep}{}{}{\theoremfont}{.}{5pt}{\thmname{#1}\thmnumber{ #2}\thmnote{ #3}}

\theoremstyle{Zwei}
\newtheorem{thm}{Theorem}[section]

\newcommand{\Projeins}{\ensuremath{\operatorname{Proj}^{\,\operatorname{1}}}}
\newcommand{\ind}[1]{\operatorname{ind}_{#1}\nolimits}
\newcommand{\proj}[1]{\operatorname{proj}_{#1}\nolimits}
\newcommand{\Bigsum}[2]{\ensuremath{\mathop{\textstyle\sum}_{#1}^{#2}}}

\allowdisplaybreaks

\newcommand{\spacea}{\hspace{1.2pt}}

\usepackage[pdftex]{color}

\begin{document}
\title{PLB-spaces of holomorphic functions with logarithmic growth conditions}
\author{S.-A.~Wegner\,$^{\flat}$}
\date{November 23, 2011}
\maketitle
\renewcommand{\thefootnote}{}
\hspace{-1000pt}\footnote{$^{\flat}$\hspace{1pt}Sven-Ake Wegner, Fachbereich C -- Mathematik und Naturwissenschaften, Arbeitsgruppe Funktionalanalysis, Bergische Universit\"at Wuppertal, Gau\ss{}\-stra\ss{}e 20, D-42097 Wuppertal, Germany, Phone:\spacea{}\spacea{}+49\spacea{}(0)\spacea{}202\spacea{}/\spacea{}439\spacea{}-\spacea{}2531, e-mail: wegner@math.uni-wuppertal.de.\vspace{3pt}}
\hspace{-1000pt}\footnote{\hspace{4.5pt}2010 \emph{Mathematics Subject Classification}: Primary 46E10; Secondary 46A13.}
\hspace{-1000pt}\footnote{\hspace{4.5pt}\emph{Key words and phrases}: PLB-space, derived projective limit functor, weighted space.}
\vspace{-45pt}
{\small
\begin{abstract}\noindent{}Countable projective limits of countable inductive limits, called PLB-spaces, of weighted Banach spaces of continuous functions have recently been investigated by Agethen, Bierstedt and Bonet. In a previous article, the author extended their investigation to the case of holomorphic functions and characterized when spaces over the unit disc w.r.t.~weights whose decay, roughly speaking, is neither faster nor slower than that of a polynomial are ultrabornological or barrelled. In this note, we prove a similar characterization for the case of weights which tend to zero logarithmically.
\end{abstract}}

\section{Introduction}\label{Introduction}
\vspace{-10pt}

PLB-spaces, i.e.~countable projective limits of countable inductive limits of Banach spaces arise naturally in analysis. The space of distributions, the space of real analytic functions and several spaces of ultradistributions and ultradifferentiable functions constitute prominent examples for spaces of this type. In fact, all these spaces turn out to be even PLS-spaces, which means that the linking maps in the inductive spectra of Banach spaces are not only continuous but even compact. During the last years the theory of PLS-spaces has played an important role in the application of abstract functional analytic methods to several classical problems in analysis. We refer to the survey article \cite{Domanski2004} of Doma\'{n}ski for applications, examples and further references.
\smallskip
\\A fundamental tool in the theory of PLS-spaces is the so-called first derived functor of the projective limit functor $\Projeins$. In the late sixties, Palamodov \cite{Palamodov1968} applied this concept from homological algebra to the theory of locally convex spaces. Since the mid eighties, Vogt \cite{VogtLectures} and many others intensified the research on this subject. We refer to the book of Wengenroth \cite{Wengenroth} for a systematic exposition of the theory and a detailed list of references. Among many other results, Vogt \cite{VogtLectures, Vogt1989}, see \cite[3.3.4 and 3.3.6]{Wengenroth}, proved that there is a connection between the vanishing of $\Projeins$ on a countable projective spectrum of LB-spaces and locally convex properties of the corresponding projective limit (e.g.~being ultrabornological or barrelled). In addition, Vogt \cite[Section 4]{Vogt1989} gave characterizations of the vanishing of $\Projeins$ and the forementioned properties in the case of sequence spaces. Recently, 
Agethen, Bierstedt, Bonet \cite{ABB2009} extended his results to weighted PLB-spaces of continuous functions. In \cite{WegnerJMAA} the author investigated weighted PLB-spaces of holomorphic functions, first on arbitrary balanced domains in $\mathbb{C}^d$ and then on the unit disc with certain technical assumptions on the weights which, roughly speaking, mean that the weights tend to zero like a polynomial.
\smallskip
\\In this article we restrict ourselves to the unit disc right from the beginning and consider weights which tend to zero logarithmically (see the remarks at the end of Section \ref{Notation:sec}). More formally, we consider weights satisfying the so-called condition (LOG) invented by Bonet, Engli\v{s}, Taskinen \cite{BoEnTa2005}. In Section \ref{Notation:sec} we state its definition and also the formal definition of the PLB-space $AH(\mathbb{D})$ which is the object of our study. In Section \ref{UaB} we recall several weight conditions due to Vogt \cite{Vogt1992} as well as variants introduced in \cite{WegnerJMAA}. Then we present the first result, Theorem \ref{LOG-main-result}, which provides necessary and sufficient conditions for the vanishing of $\Projeins$ and for $AH(\mathbb{D})$ being ultrabornological and barrelled. In Section \ref{Inter} we then treat the question of interchangeability of projective and inductive limit in the definition of $AH(\mathbb{D})$, where we also are able  to prove 
necessary and sufficient conditions, see Theorem \ref{Int-result}, using a condition due to Vogt \cite{Vogt1983} and a variant introduced in \cite{WegnerJMAA}. Since in all our results up to this point the necessary and the sufficient conditions differ, we next identify additional assumptions that allow to turn the results of the previous sections into characterizations, see Section \ref{FurtherAss}. In addition, at the end of Section \ref{FurtherAss} we make some remarks on the construction of examples.

\section{Notation}\label{Notation:sec}
\vspace{-10pt}

Let $\mathbb{D}$ denote the unit disc of the complex plane and $H(\mathbb{D})$ the space of all holomorphic functions on $\mathbb{D}$ endowed with the topology co of uniform convergence on the compact subsets. Let $\mathcal{A}=((a_{N,n})_{N\in\mathbb{N}})_{n\in\mathbb{N}}$ be a double sequence of strictly positive and continuous functions (weights) on $\mathbb{D}$ which is decreasing in $n$ and increasing in $N$, i.e.~$a_{N,n+1}\leqslant{}a_{N,n}\leqslant{}a_{N+1,n}$ holds for all $N$ and $n$; this condition will be assumed on the double sequence $\mathcal{A}$ in the remainder of this article. We define
\begin{align*}
Ha_{N,n}(\mathbb{D})&:=\{\,f\in H(\mathbb{D})\: ; \: \|f\|_{N,n}:=\sup_{z\in\mathbb{D}}a_{N,n}(z)|f(z)|<\infty\,\},
\end{align*}
which is a Banach space for the norm $\|\cdot\|_{N,n}$ whose closed unit ball we denote by $B_{N,n}$. By definition, $Ha_{N,n}(\mathbb{D})\subseteq Ha_{N,n+1}(\mathbb{D})$ holds with continuous inclusion for all $N$ and $n$ and we can define for each $N$ the weighted inductive limit
$$
\mathcal{A}_NH(\mathbb{D}):=\ind{n}Ha_{N,n}(\mathbb{D})
$$
which is a complete and hence regular LB-space by Bierstedt, Meise, Summers \cite[end of the remark after 1.13]{BMS1982}. For each $N$ we have $\mathcal{A}_{N+1}H(\mathbb{D})\!\subseteq\!\mathcal{A}_NH(\mathbb{D})$ with continuous inclusion. Hence, $\mathcal{A}H\!:=\!(\mathcal{A}_NH(\mathbb{D}))_{N\in\mathbb{N}}$ is a projective spectrum of LB-spaces with inclusions as linking maps and we can now form the following projective limit, called \textit{weighted PLB-space of holomorphic functions}
$$
AH(\mathbb{D}):=\proj{N}\mathcal{A}_NH(\mathbb{D})=\ind{n}\proj{N}Ha_{N,n}(\mathbb{D})
$$
which is the object of our study in this work. We refer the reader to the book of Wengenroth \cite{Wengenroth} for a detailed exposition of the theory of projective spectra of locally convex spaces $\mathcal{X}=(X_N)_{N\in\mathbb{N}}$, their projective limits $\proj{N}X_N$ and the derived functor $\Projeins$.
\smallskip
\\In this article we consider weights of the following special type; the definition of the so-called condition (LOG) is due to Bonet, Engli\v{s}, Taskinen \cite[4.1]{BoEnTa2005} and was used to prove a projective description for weighted LB-spaces of holomorphic functions. For every $\kappa\in\mathbb{N}$ we put $r_{\kappa}:=1-2^{-2^{\kappa}}$, $r_0:=0$ and $I_{\kappa}:=[r_{\kappa},r_{\kappa+1}]$. We say that the sequence $\mathcal{A}=((a_{N,n})_{N\in\mathbb{N}})_{n\in\mathbb{N}}$ satisfies condition (LOG) if each weight in the sequence is radial and approaches monotonically zero as $r\nearrow1$ and if there exist constants $0<a<1<A$ such that                                                                                           
$$
\text{(LOG 1)}\;\;A\cdot a_{N,n}(r_{\kappa+1})\geqslant a_{N,n}(r_{\kappa})\;\text{ and\; (LOG 2)}\;\;a_{N,n}(r_{\kappa+1})\leqslant a\cdot a_{N,n}(r_{\kappa})
$$
hold for all $N$, $n$ and $\kappa\in\mathbb{N}$. For our investigation we need the following well-known fact; for a proof we refer to \cite[Proof of Remark 1.1]{WegnerBET}: If $v$ is a radial weight which is decreasing on  $[0,1[$ and $(r_n)_{n\in\mathbb{N}}\subseteq[0,1[$ a sequence with $r_n\nearrow1$ as $n\rightarrow\infty$, then for $g\in Hv(\mathbb{D})$ we have $g_n\rightarrow g$ w.r.t.~co, where $g_n(z):=g(r_nz)$ for $z\in\mathbb{D}$.
\smallskip
\\Let us add the following explanatory comment on the meaning of condition (LOG). We assume that $v\colon\mathbb{D}\rightarrow\:]0,\infty[$ satisfies (LOG), put $K:=\{1/(1-r_{\kappa}) \:;\:\kappa\in\mathbb{N}\}$ and consider $w\colon K\rightarrow\,]0,\infty[$ with $w(1/(1-r_{\kappa}))=v(r_{\kappa})$. Then $v(0)(1/A)^{\log\log \nu}\leqslant w(\nu)\leqslant v(0)a^{\log\log\nu}$ holds for all $\nu\in K\backslash\{1\}$, where $\log$ denotes the binary logarithm. Therefore, $w(\nu)=\mathcal{O}(1/(\log \nu)^c)$ and $w(\nu)=\Omega(1/(\log \nu)^C)$ is valid for $\nu\rightarrow\infty$, where $C=-\log1/A>0$ and $c=-\log a>0$. In this sense, the weights satisfying (LOG) tend to zero logarithmically.

\section{Proj$^{\text{\,1}}$ and locally convex properties}\label{UaB}
\vspace{-10pt}

In this section we present necessary and sufficient conditions for the vanishing of $\Projeins\mathcal{A}H$ and $AH(\mathbb{D})$ being ultrabornological and barrelled. To formulate our results, we need the following notation due to Vogt \cite{Vogt1992}. We say that the sequence $\mathcal{A}$ satisfies condition (Q) if 
$$
\textstyle\forall \: N \; \exists \: M \geqslant N,\, n \; \forall \: K \geqslant M,\, m,\, \varepsilon > 0 \; \exists \: k, \, S>0 : \frac{1}{a_{M,m}} \leqslant \max\big(\frac{\varepsilon}{a_{N,n}},\frac{S}{a_{K,k}}\big),
$$
we say that it satisfies (wQ) if
$$
\textstyle\forall \: N \; \exists \: M \geqslant N,\, n \; \forall \: K \geqslant M,\, m \; \exists \: k, \, S>0 : \frac{1}{a_{M,m}} \leqslant \max\big(\frac{S}{a_{N,n}},\frac{S}{a_{K,k}}\big).
$$
Condition (Q) implies condition (wQ); the converse is not true, cf.~Bierstedt, Bonet \cite{BB1994}. We define condition $\text{(Q)}^{\sim}$ by the same quantifiers as in (Q) but the estimate replaced with $(a_{M,m}^{-1})^{\sim}\leqslant\max(\varepsilon a_{N,n}^{-1},S a_{K,k}^{-1})^{\sim}$, where for a given weight $a$, $(1/a)^{\sim}\colon \mathbb{D}\rightarrow \mathbb{R}$, $z\mapsto\sup\{|g(z)|\:;\:g\in H(\mathbb{D}),\,a|g|\leqslant 1\}$ is the associated growth condition and $\tilde{a}:=1/(1/a)^\sim$ the associated weight, cf.~Bierstedt, Bonet, Taskinen \cite{BBT}. We define $\text{(wQ)}^{\sim}$ by the same quantifiers as in (wQ) and the estimate replaced with $(a_{M,m}^{-1})^{\sim}\leqslant\max(\varepsilon(a_{N,n}^{-1})^{\sim},S(a_{K,k}^{-1})^{\sim})$.
\smallskip
\\Now we are able to present our first result;  its proof was inspired by the method developed in \cite[Section 4]{BoEnTa2005}, see also \cite{WegnerBET}.

\begin{thm}\label{LOG-main-result} Let $\mathcal{A}$ satisfy condition (LOG). Then we have the implications (i)$\Rightarrow$(ii) $\Rightarrow$(iii)$\Rightarrow$(iv)$\Rightarrow$(v), where
\vspace{5pt}
\\\begin{tabular}{rlrl}%\vspace{4pt}
(i)   & $\mathcal{A}$ satisfies condition $\text{(Q)}^{\sim}$,            & (iv) & $AH(\mathbb{D})$ is barrelled, \\
(ii)  & $\Projeins \mathcal{A}H=0$,                                                    & (v)  & $\mathcal{A}$ satisfies condition $\text{(wQ)}^{\sim}$. \\
(iii) & $AH(\mathbb{D})$ is ultrabornological,                                         &      & \\
\end{tabular}
\end{thm}
\begin{proof}\textquotedblleft{}(i)$\Rightarrow$(ii)\textquotedblright{}\; In order to show that $\Projeins \mathcal{A}H=0$ holds, we use Braun, Vogt \cite[Theorem 8]{BraunVogt1997} -- which was obtained independently by Frerick, Wengenroth \cite{FrerickWengenroth1996}. That is, we have to verify condition $\text{(}\overline{\text{P}_{\!2}}\text{)}$
$$
\forall \: N \; \exists \: M,\,n \; \forall \: K,\,m,\,\varepsilon>0 \; \exists \: k,\,S>0\: \colon \: B_{M,m}\subseteq \varepsilon B_{N,n}+SB_{K,k}.
$$
We denote by $0<a<1<A$ the constants of (LOG 1) and (LOG 2) and put $B:=\max\big(\sum_{\kappa=0}^{\infty}a^{\kappa},\,\sup_{\kappa>t+2}2^{-\kappa}A^{\kappa-t}2^{-2^{\kappa-1}}\big)$. In addition, we put $T:=2A^2(B+A^2)+ 4(A^2+2B)$.
\smallskip
\\For given $N$ we select $M$ and $n$ as in $\text{(Q)}^{\sim}$. For given $K,\,m,\,\varepsilon>0$ we put $\varepsilon':=\frac{\varepsilon}{2T}$ and choose $k$ and $S'>0$ according to $\text{(Q)}^{\sim}$ w.r.t.~$\varepsilon'$ and put $S:=2TS'$. Now we fix an arbitrary $f\in B_{M,m}$. We have $|f|\leqslant\frac{1}{a_{M,m}}$, i.e.~with \cite[1.2.(iii)]{BBT} it follows $|f|\leqslant (\frac{1}{a_{M,m}})^{\sim}$. By the estimate in $\text{(Q)}^{\sim}$ we obtain $|f|\leqslant\max(\varepsilon'(\frac{1}{a_{N,n}}),\,S'(\frac{1}{a_{K,k}}))^{\sim}\leqslant\max(\frac{\varepsilon'}{a_{N,n}},\frac{S'}{a_{K,k}})$ where the last estimate follows from \cite[1.2.(i)]{BBT}. We put $u:=\min(\frac{a_{N,n}}{\varepsilon'},\,{\textstyle\frac{a_{K,k}}{S'}})$. Hence $\|f\|_u:=\sup_{z\in\mathbb{D}}u(z)|f(z)|\leqslant1$. By defining $u_0:=a_{N,n}$, $u_1:=a_{K,k}$, $a_0:=1/\varepsilon'$ and $a_1:=1/S'$ we get $u=\min(a_0u_0,\,a_1u_1)$. We put (according to \cite[Proof of 4.5]{BoEnTa2005}) $f_{r_{\kappa}}(z):=f(r_{\kappa}z)$ and obtain 
$f_{r_{\kappa}}\rightarrow f$ within the compact open topology (cf.~our remarks after the definition of (LOG) in Section \ref{Notation:sec}).
\medskip
\\Since all the weights in $\mathcal{A}$ are non-increasing, this is also true for $u$. Hence
$$
\text{(1)}\;\;\;\;\;\;\inf_{|z|\in I_{\kappa}}u(z)\:=\:u(r_{\kappa+1})\:\geqslant\:u(r_{\kappa+2})\:\,=\hspace{-5pt}\inf_{|z|\in I_{\kappa+1}}\!\!u(z)\!\!\stackrel{\scriptscriptstyle\text{(LOG 1)}}{\geqslant}\!\!A^{-2}u(r_{\kappa})
$$
holds. For every $\kappa$ in $\mathbb{N}$ we pick $i(\kappa)\in\{0,\,1\}$ such that
$$
\text{(2)}\;\;\;\;\;\;u(r_{\kappa})=a_{i(\kappa)}u_{i(\kappa)}(r_{\kappa})=a_{i(\kappa)}\sup_{|z|\in I_{\kappa}}u_{i(\kappa)}(z)
$$
is valid. For $\nu\in\mathbb{N}$ and $\ell\in\{0,\,1\}$ we define $N_{\ell}:=\{\,\kappa\in\mathbb{N}\:;\:\kappa\leqslant \nu\text{ and } i(\kappa)=\ell\,\}$. For each $\kappa\geqslant1$ we put $g_{\kappa}(z):=f(r_{\kappa+1}z)-f(r_{\kappa}z)$ and $g_0(z):=f(0)$. Finally, we define
$$
\text{(3)}\;\;\;\;\;\;h_{\ell}\,:=\!\Bigsum{\kappa\in N_{\ell}}{}g_{\kappa}
$$
for $\ell\in\{0,1\}$ and compute $f_{r_{\nu+1}}=g_0+h_0+h_1$. For the constant function $g_0$ we have $|g_0(z)|=|f(0)|=|f(r_0)|\leqslant a_{i(0)}^{-1}u_{i(0)}(0)^{-1}$ that is $a_{N,n}(z)|g_0(z)|\leqslant\varepsilon'\leqslant\frac{\varepsilon}{2}$ (if $i(0)=0$) or $a_{K,k}(z)|g_0(z)|\leqslant S'\leqslant{\textstyle\frac{S}{2}}$ (if $i(0)=1$). Now we fix $\ell\in\{0,1\}$, pick $\kappa\in N_{\ell}$ and estimate $|g_{\kappa}(z)|$ for different $z$.\vspace{3pt}
\begin{compactitem}
\item[1.] Assume first $|z|\geqslant r_{\kappa-1}$ (where we put $r_{\kappa-1}:=r_0$ for $\kappa=0$).\vspace{3pt}
\begin{compactitem}
\item[a.] Let $\kappa\geqslant2$. Then we have $|r_{\kappa}z|\geqslant(1-2^{-2^{\kappa}})(1-2^{-2^{\kappa-1}})\geqslant r_{\kappa-2}$ and thus $r_{\kappa-2}\leqslant|r_{\kappa}z|\leqslant |r_{\kappa+1}z|\leqslant r_{\kappa+1}$. Since $\|f\|_u\leqslant1$, we have $|f(z)|\leqslant u(z)^{-1}$ on $\mathbb{D}$. Since $u$ is non-increasing and by (1) we get the estimate $|g_{\kappa}(z)|\leqslant{}|f(r_{\kappa}z)|+|f(r_{\kappa+1}z)|\leqslant{}2\,\max\,(\,\sup_{r\in I_{\kappa-2}}u(r)^{-1},\:\sup_{r\in I_{\kappa-1}}u(r)^{-1},\linebreak{}\:\sup_{r\in I_{\kappa}}u(r)^{-1}) \leqslant 2u(r_{\kappa+1})^{-1}\leqslant 2 A^2 u(r_{\kappa})^{-1}= 2A^2 a_{\ell}^{-1}u_{\ell}(r_{\kappa})^{-1}$ where last equality follows since $u(r_{\kappa})=a_{i(\kappa)}u_{i(\kappa)}(r_{\kappa})$ and $\kappa\in N_{\ell}$ implies $i(\kappa)=\ell$ (cf.~(2)).\vspace{3pt}
\item[b.] Let $\kappa=1$. In this case we have $|g_1(z)|\leqslant |f(r_2z)|+|f(r_1z)|\leqslant2\linebreak{}\sup_{r_0\leqslant r\leqslant r_2}u(r)^{-1}\leqslant 2\max(\sup_{r\in I_0}u(r)^{-1},\,\sup_{r\in I_1}u(r)^{-1})=2u(r_2)^{-1}\leqslant2A^2a_{\ell}^{-1}u_{\ell}(r_1)^{-1}$\! where we use (1) for the last estimate.
\vspace{3pt}
\item[c.] Let $\kappa=0$. We have $|g_{\kappa}(z)|=|f(0)|$ and $\|f\|_u\leqslant1$ implies in particular $u(0)|f(0)|\leqslant1$, i.e.~
$
|g_{\kappa}(z)|=|f(0)|  \leqslant u(0)^{-1}= u(r_0)^{-1}=a_{i(0)}^{-1}u_{i(0)}(r_0)^{-1}\leqslant 2A^2a_{i(\kappa)}^{-1}u_{i(\kappa)}(r_{\kappa})^{-1}=2A^2a_{\ell}^{-1}u_{\ell}(r_{\kappa})^{-1}
$
by (2), since $A>1$ and by our selection $\kappa\in N_{\ell}$.\vspace{3pt}
\end{compactitem}
To sum up, in case 1.~we have
$$
\text{(4)}\;\;\;\;\;\;|g_{\kappa}(z)|\leqslant 2A^2 a_{\ell}^{-1}u_{\ell}(r_{\kappa})^{-1}
$$
for $|z|\geqslant r_{\kappa-1}$ and $\kappa\geqslant0$.\vspace{5pt}
\item[2.] Assume that $\kappa>t+1$ and $|z|\in I_t$, i.e.~$r_t\leqslant|z|\leqslant r_{t+1}$. We have $|g_{\kappa}(z)|=|f(r_{\kappa}z)-f(r_{\kappa+1}z)|$ by definition. By the mean value theorem there exists $\xi$ between $r_{\kappa}z$ and $r_{\kappa+1}z$ with $|f(r_{\kappa}z)-f(r_{\kappa+1})|=|f'(\xi)||r_{\kappa}z-r_{\kappa+1}z|\leqslant|f'(\xi)||r_{\kappa}-r_{\kappa+1}|$. Since $|r_{\kappa+1}-r_{\kappa}|\leqslant 2^{-2^{\kappa}}$ holds, the above yields $|g_{\kappa}(z)|\leqslant\sup_{r_{\kappa}r_t\leqslant|\xi|\leqslant r_{\kappa+1}r_{t+1}}|f'(\xi)|2^{-2^{\kappa}}$. Our assumption $t<\kappa-1$ implies $|\xi|\leqslant r_{\kappa+1}r_{t+1}< r_{t+1}\leqslant r_{\kappa}$ and we thus may use the Cauchy formula
$$
\text{(5)}\;\;\;\;\;\;|f'(\xi)|\leqslant{\textstyle\frac{1}{2\pi}}\int_{|\eta|=r_{\kappa}}{\textstyle\frac{|f(\eta)|}{|\eta-\xi|^2}}d\eta
$$
to estimate $|f'(\xi)|$. We have $|f(\eta)|\leqslant u(\eta)^{-1}=u(r_{\kappa})^{-1}$, since $\|f\|_u\leqslant1$ and $u$ is radial. For $|\eta|=r_{\kappa}$ and $\kappa>t+2$ straightforward computations show $\frac{1}{|\eta-\xi|^2}\leqslant4\cdot2^{2^{\kappa-1}}$, i.e.~$|f'(\xi)|\leqslant{\textstyle\frac{2\pi r_{\kappa}}{2\pi}}\cdot4\cdot2^{2^{\kappa-1}}u(r_{\kappa})^{-1}\leqslant 4\cdot2^{2^{\kappa-1}}u(r_{\kappa})^{-1}$ holds by (5) and we get 
$|g_{\kappa}(z)|\leqslant4\cdot2^{2^{\kappa-1}}a_{\ell}^{-1}u_{\ell}(r_{\kappa})^{-1}$. If $\kappa=t+2$, similar computations show $|g_{\kappa}(z)|\leqslant 4a_{\ell}^{-1}u_{\ell}(r_{\kappa})^{-1}$. Now we use (LOG 1) $(\kappa-t)-$times to obtain $u_{\ell}(r_t)\leqslant Au_{\ell}(r_{\kappa+1})\leqslant A^2u_{\ell}(r_{t+2})\leqslant\cdots\leqslant A^{\kappa-t}u_{\ell}(r_{t+\kappa-t})\leqslant A^{\kappa-t}u_{\ell}(r_{\kappa})$. Since $|z|\geqslant r_t$ and since $u_{\ell}$ is radial and decreasing for $r\nearrow 1$ we have $u_{\ell}(r_t)\geqslant u_{\ell}(z)$ and thus we get $u_{\ell}(z)\leqslant u_{\ell}(r_{t})\leqslant A^{\kappa-t}u_{\ell}(r_{\kappa})$, which finally yields $u_{\ell}(r_{\kappa})^{-1}\leqslant A^{\kappa-t}u_{\ell}(z)^{-1}$. For $\kappa>t+2$ we get $|g_{\kappa}(z)|\leqslant 4a_{\ell}^{-1}u_{\ell}(z)^{-1}A^{\kappa-t}2^{-2^{\kappa-1}}$ from the latter and thus $|g_{\kappa}(z)|\leqslant4\cdot2^{-\kappa}Ba_{\ell}^{-1}u_{\ell}(z)^{-1}$ by our selection of $B$. If $\kappa=t+2$ we get $|g_{\kappa}(z)
|\leqslant 4a_{\ell}^{-1}u_{\ell}(z)^{-1}A^{2}$.
\smallskip
To sum up, in case 2.~we have
$$
\text{(6)}\;\;\;\;\;\;|g_{\kappa}(z)|\leqslant 4a_{\ell}^{-1}u_{\ell}(z)^{-1}\,
\begin{cases}
\;2^{-\kappa}B & \text{if }\;\kappa>t+2\\
\;\;\,A^{2} & \text{if }\;\kappa=t+2
\end{cases}
$$
for $|z|\in I_t$ and $\kappa$ as indicated above.
\end{compactitem}
To complete the proof, let now $z\in\mathbb{D}$ be arbitrary. We select $t\in\mathbb{N}$ such that $|z|\in I_t=[r_t,r_{t+1}]$. Then
$$
\text{(7)}\;\;\;\;|h_{\ell}(z)|\,\stackrel{\text{\tiny dfn}}{=}\,\big|\!\!\Bigsum{\kappa\in N_{\ell}}{}\!g_{\kappa}(z)\hspace{0.5pt}\big|\,\leqslant\!\Bigsum{\stackrel{\kappa\in N_{\ell}}{\scriptscriptstyle\kappa\leqslant t+1}}{}\hspace{-2pt}|g_{\kappa}(z)|\:+\!\!\Bigsum{\stackrel{\kappa\in N_{\ell}}{\scriptscriptstyle\kappa>t+1}}{}\hspace{-2pt}|g_{\kappa}(z)|\:=:\:G_{\ell}(z)+H_{\ell}(z).
$$
\begin{compactitem}
\item[(i)] Consider $G_{\ell}(z)$, that is all occurring $\kappa$ satisfy $0\leqslant\kappa\leqslant t+1$ and $\kappa\in N_{\ell}$. Thus we have $\kappa-1\leqslant t$, hence $|z|\geqslant r_t\geqslant r_{\kappa-1}$ (remember that we defined $r_{-1}:=r_0=0$). By (4) we therefore have
$$
\text{(8)}\;\;\;\;\;\;G_{\ell}(z)\,\stackrel{\text{\tiny dfn}}{=}\!\!\Bigsum{\stackrel{\kappa\in N_{\ell}}{\scriptscriptstyle\kappa\leqslant t+1}}{}\hspace{-2pt}|g_{\kappa}(z)|\:\leqslant\!\Bigsum{\stackrel{\kappa\in N_{\ell}}{\scriptscriptstyle\kappa\leqslant t+1}}{}\hspace{-1.5pt}2A^2a_{\ell}^{-1}u_{\ell}(r_{\kappa})^{-1}.
$$
(LOG 2) implies $u_{\ell}(r_{\kappa+1})\leqslant a u_{\ell}(r_{\kappa})$, i.e.~$u_{\ell}(r_{\kappa})^{-1}\leqslant a u_{\ell}(r_{\kappa+1})^{-1}$ for arbitrary $\kappa$. Iterating this estimate $t-\kappa$ times for a fixed $\kappa\leqslant t$ we get $u_{\ell}(r_{\kappa})^{-1}\leqslant a u_{\ell}(r_{\kappa+1})^{-1}\leqslant\cdots\leqslant a^{t-\kappa}u_{\ell}(r_{\kappa+t-\kappa})^{-1}=a^{t-\kappa}u_{\ell}(r_{t})^{-1}$.
With the latter we get from (8) 
\begin{align*}
G_{\ell}(z)
& \leqslant 2A^2a_{\ell}^{-1}\big(\Bigsum{\kappa=0}{t}u_{\ell}(r_{\kappa})^{-1} + u_{\ell}(r_{t+1})^{-1}\big) \\
&\leqslant 2A^2a_{\ell}^{-1}\big(\Bigsum{\kappa=0}{t}a^{t-\kappa}u_{\ell}(r_t)^{-1} + A^2u_{\ell}(r_{t})^{-1}\big) \\
&= 2A^2a_{\ell}^{-1}u_{\ell}(r_t)^{-1}\big(\Bigsum{\sigma=0}{t}a^{\sigma} + A^2\big) \\
&\leqslant 2A^2(B+A^2)a_{\ell}^{-1}u_{\ell}(z)^{-1}
\end{align*}
where we used that $B>\Bigsum{\kappa\in\mathbb{N}}{}a^{\kappa}$, that $u_{\ell}$ is radial and decreasing for $r\nearrow1$ and $|z|\geqslant r_t$, whence $u_{\ell}(r_t)^{-1}\leqslant u_{\ell}(z)^{-1}$.\vspace{3pt}
\item[(ii)] Consider $H_{\ell}(z)$. Then all the occurring $\kappa$ satisfy $\kappa>t+1$ and $\kappa\in N_{\ell}$. By (6) we obtain
\begin{align*}
\;\;\;\;\;\;\;H_{\ell}(z)\,\stackrel{\text{\tiny dfn}}{=}\hspace{-1.8pt}\Bigsum{\stackrel{\kappa\in N_{\ell}}{\scriptscriptstyle\kappa>t+1}}{}\hspace{-2pt}|g_{\kappa}(z)|\;&=\;\delta_{i(t+2),\ell}\,|g_{t+2}|\:+\hspace{-2pt}\Bigsum{\stackrel{\kappa\in N_{\ell}}{\scriptscriptstyle\kappa>t+2}}{}\hspace{-1.5pt}|g_{\kappa}(z)|\\
&\leqslant\;4a_{\ell}^{-1}u_{\ell}(z)^{-1}A^2\,+\hspace{-2pt}\Bigsum{\stackrel{\kappa\in N_{\ell}}{\scriptscriptstyle\kappa>t+2}}{}\!4\cdot2^{-\kappa}Ba_{\ell}^{-1}u_{\ell}(z)^{-1}\\
&\leqslant\;\big(4A^2+4B\Bigsum{\kappa=0}{\infty}2^{-\kappa}\big)a_{\ell}^{-1}u_{\ell}(z)^{-1}\\
&=\;4(A^2+2B)a_{\ell}^{-1}u_{\ell}(z)^{-1},
\end{align*}
where $\delta$ denotes the Kronecker symbol.\vspace{3pt}
\end{compactitem}
Combining the estimates in (i) and (ii) we obtain $|h_{\ell}(z)|\leqslant G_{\ell}(z)+H_{\ell}(z)\leqslant(2A^2(B+A^2)+ 4(A^2+2B))a_{\ell}^{-1}u_{\ell}(z)^{-1}=Ta_{\ell}^{-1}u_{\ell}(z)^{-1}$ from (7), that means $u_{\ell}(z)|h_{\ell}(z)|\leqslant Ta_{\ell}^{-1}$ for each $z\in\mathbb{D}$ and $\ell=0,\,1$. By the definition of $u_{\ell}$ and $a_{\ell}$ this means $a_{N,n}(z)|h_0(z)|=u_0(z)|h_0(z)|\leqslant T\varepsilon'\leqslant{\textstyle\frac{\varepsilon}{2}}$ and $a_{K,k}(z)|h_1(z)|=u_1(z)|h_1(z)|\leqslant TS'\leqslant{\textstyle\frac{S}{2}}$ for each $z\in\mathbb{D}$. Hence $h_0\in \frac{\varepsilon}{2}B_{N,n}$ by the first estimate and $h_1\in \frac{S}{2}B_{K,k}$ by the second estimate, i.e.~$f_{r_{\nu+1}}=g_0+h_0+h_1\in{\textstyle\frac{\varepsilon}{2}}B_{N,n}+{\textstyle\frac{S}{2}}B_{K,k}+{\textstyle\frac{\varepsilon}{2}}B_{N,n}+{\textstyle\frac{S}{2}}B_{K,k}\subseteq\varepsilon B_{N,n} + S B_{K,k}$. Since $\varepsilon B_{N,n} + S B_{K,k}$ is co-compact, $f\in \varepsilon B_{N,n} + S B_{K,k}$ follows 
from our remark after the definition of (LOG) and we are done.
\smallskip
\\\textquotedblleft{}(ii)$\Rightarrow$(iii)$\Rightarrow$(iv)\textquotedblright{}\; The first implication holds in general, see Wengenroth \cite[3.3.4]{Wengenroth} (cf.~Vogt \cite[5.7]{VogtLectures}). The second also holds in general, see e.g.~\cite{MeiseVogtEnglisch}.
\smallskip
\\\textquotedblleft{}(iv)$\Rightarrow$(v)\textquotedblright{}\; The assumptions of (LOG) imply that the topology of $Ha_{N,n}(\mathbb{D})$ is stronger than co and that the polynomials are contained in the latter space for any $N$ and $n$; by Bierstedt, Bonet, Galbis \cite[remark previous to 1.2]{BBG} the second statement is equivalent to requiring that each weight $a_{N,n}$ extends continuously to $\bar{\mathbb{D}}$ with $a_{N,n}|_{\partial\mathbb{D}}=0$. Therefore, the statement follows from \cite[Section 3]{WegnerJMAA}.
\end{proof}

\section{Interchangeability of projective and inductive limit}\label{Inter}
\vspace{-10pt}

Given a sequence of weights $\mathcal{A}=((a_{N,n})_{N\in\mathbb{N}})_{n\in\mathbb{N}}$ on the unit disc we can -- in addition to the PLB-space investigated in the preceding section -- also associate a \textit{weighted LF-space of holomorphic functions} by defining
\begin{align*}
\mathcal{V}H(\mathbb{D})&:=\ind{n}\proj{N}Ha_{N,n}(\mathbb{D}).
\end{align*}
This space constitutes the holomorphic version of the weighted LF-space of continuous functions investigated by Bierstedt, Bonet \cite{BB1994} and has been studied by several authors in different contexts, cf.~the survey \cite{Bierstedt2001} of Bierstedt for detailed references. In \cite[Section 3]{WegnerJMAA} the author showed that $\mathcal{V}H(\mathbb{D})\subseteq AH(\mathbb{D})$ holds in general with continuous inclusion. In order to investigate when $\mathcal{V}H(\mathbb{D})= AH(\mathbb{D})$ holds algebraically, in \cite{WegnerJMAA} the following condition based on research of Vogt \cite[1.1]{Vogt1983} was used: We say that a sequence $\mathcal{A}$ as above satisfies condition (B) if
$$
\forall \: (n(N))_{N\in\mathbb{N}}\subseteq \mathbb{N} \; \exists \: m \; \forall \: M \; \exists \: L,\, c>0 \colon a_{M,m} \, \leqslant \, c \max_{\scriptscriptstyle N=1,\dots,L} a_{N,n(N)}.
$$
Condition $\text{(B)}^{\sim}$ is defined by the same quantifiers and the estimate replaced by $\tilde{a}_{M,m} \, \leqslant \, c (\max_{\scriptscriptstyle N=1,\dots,L} a_{N,n(N)})^{\sim}$. Now we are ready to state the result on interchangeability of projective and inductive limit.

\begin{thm}\label{Int-result}Let $\mathcal{A}$ satisfy condition (LOG).
\begin{compactitem}\vspace{2pt}
\item[(i)] $AH(\mathbb{D})=\mathcal{V}H(\mathbb{D})$ holds algebraically if and only if $\mathcal{A}$ satisfies $\text{(B)}^{\sim}$.\vspace{3pt}
\item[(ii)] If $AH(\mathbb{D})=\mathcal{V}H(\mathbb{D})$ holds algebraically and topologically then $\mathcal{A}$ satisfies the conditions $\text{(B)}^{\sim}$ and $\text{(wQ)}^{\sim}$.\vspace{3pt}
\item[(iii)] If $\mathcal{A}$ satisfies the conditions $\text{(B)}^{\sim}$ and $\text{(Q)}^{\sim}$ then $AH(\mathbb{D})=\mathcal{V}H(\mathbb{D})$ holds algebraically and topologically.
\end{compactitem}
\end{thm}
\begin{proof}(i) As in the proof of the implication \textquotedblleft{}(iv)$\Rightarrow$(v)\textquotedblright{} of Theorem \ref{LOG-main-result} the statement is even true in a more general setting, see \cite[Section 5]{WegnerJMAA}.
\smallskip
\\(ii) This follows directly from (i) and Theorem \ref{LOG-main-result}.
\smallskip
\\(iii) Let $\text{(B)}^{\sim}$ and $\text{(Q)}^{\sim}$ be satisfied. By (i) the identity $\mathcal{V}H(\mathbb{D})\rightarrow AH(\mathbb{D})$ is one-to-one and as we noted above it is continuous. Since $AH(\mathbb{D})$ is ultrabornological by Theorem \ref{LOG-main-result} and $\mathcal{V}H(\mathbb{D})$ is webbed, we can apply the open mapping theorem (cf.~Meise, Vogt \cite[24.30]{MeiseVogtEnglisch}) and are done.
\end{proof}

\section{Further assumptions on the defining double sequence and some remarks on examples}\label{FurtherAss}
\vspace{-10pt}

In the previous sections we established necessary and sufficient conditions for the vanishing of $\Projeins$, for $AH(\mathbb{D})$ being ultrabornological and barrelled and also for the interchangeability of projective and inductive limit. Unfortunately, in almost all results the necessary and the sufficient conditions are distinct. Therefore it is desirable to identify additional general assumptions on the sequence $\mathcal{A}$ which allow to prove characterizations. In the sequel we present a setting in which this is possible and in which we in addition get by on associated weights which is -- in view of the complexity of our conditions -- clearly also desirable. In order to do so, we need the following condition $(\Sigma)$, which was introduced by Bierstedt, Bonet \cite[Section 5]{BB1994} and constitutes a generalization of condition (V) invented by Bierstedt, Meise, Summers \cite{BMS1982}. We say that a double sequence $\mathcal{A}=((a_{N,n})_{N\in\mathbb{N}})_{n\in\mathbb{N}}$ on $\mathbb{D}$ satisfies 
condition $(\Sigma)$ if
$$
\textstyle\forall\:N\;\exists\: K\geqslant N\:\forall\:k\;\exists\:n\geqslant k\colon \frac{a_{N,n}}{a_{K,k}} \text{ vanishes at } \infty \text{ on } \mathbb{D}.
$$
Moreover, let us -- according to Taskinen \cite{Taskinen2001} -- call a weight $a$ essential, if there exists $C>0$ such that $(1/a)^{\sim}\leqslant 1/a \leqslant C(1/a)^{\sim}$ holds. Following the lines of \cite{WegnerJMAA} we obtain the last result of this note which should be compared with the results explained in \cite{ABB2009}.

\begin{thm}\label{TopologicalResultsSigma}Let $\mathcal{A}$ satisfy the conditions (LOG) and $\text{(}{\mathit{\Sigma}}\text{)}$.\vspace{2pt}
\begin{compactitem}
\item[(1)]If all weights in $\mathcal{A}$ are essential, then the following are equivalent.\vspace{3pt}
\\\begin{tabular}{rlrl}\vspace{1pt}
(i)  & $\mathcal{A}$ satisfies (Q).                              &(iv)  & $AH(\mathbb{D})$ is barrelled. \\\vspace{1pt}
(ii) & $\Projeins \mathcal{A}H=0$.                               &(v)  & $\mathcal{A}$ satisfies (wQ). \\\vspace{1pt}
(iii)  & $AH(\mathbb{D})$ is ultrabornological.                    &      &
\end{tabular}\vspace{2pt}
\item[(2)]If $\mathcal{A}$ is contained in a set of essential weights which is closed under finite maxima, then the following are equivalent.\vspace{3pt}
\\\begin{tabular}{rlrl}\vspace{1pt}
(i)  & $\mathcal{A}$ satisfies (Q) and (B).\hspace{18.5pt} & (iii)   & $\mathcal{A}$ satisfies (wQ) and (B).\\\vspace{1pt}
\end{tabular}\vspace{-13pt}
\\\begin{tabular}{rl}
\hspace{-3pt}(ii) & $AH(\mathbb{D})=\mathcal{V}H(\mathbb{D})$ holds algebraically and topologically.
\end{tabular}
\end{compactitem}
\end{thm}
\noindent{}To conclude our investigation let us remark that examples for PLB-spaces can be constructed (according to the result \cite[7.2]{WegnerJMAA} of Bonet) as follows: Put $a_{N,n}(z)=a(|z|)^{\alpha_n}v(|z|)^{\beta_N}$ with $a$, $v\colon[0,1]\rightarrow\;]0,1]$ continuous, decreasing with $\lim_{r\nearrow1}a(r)=\lim_{r\nearrow1}v(r)=0$, $\alpha_n\nearrow\alpha\in\;]0,\infty]$ and $\beta_N\searrow\beta\in[0,\infty[$. If now the conditions (df) and (m) of \cite[7.2]{WegnerJMAA}
are both \textit{not} satisfied, then $AH(\mathbb{D})$ is a (proper) PLB-space. The latter is valid if for instance $\limsup_{r\nearrow1}\log v(r)/\log a(r)=\limsup_{r\nearrow1}\log a(r)/\log v(r)$ $=\infty$ holds. Sequences with the forementioned property can be constructed following the lines of Bierstedt, Bonet \cite[Claim on p.~765]{BB2006}.

% ------------------------------------------------------------------------
\vspace{-5pt}

\subsection*{Acknowledgment}
\vspace{-10pt}

This article arises from a part of the author's doctoral thesis, which was started at the University of Paderborn under the direction of Klaus D.~Bierstedt and after his sudden death  finished under the direction of Jos\'{e} Bonet at the Universidad Polit\'{e}cnica de Valencia. The author is greatly indebted to both of his supervisors; in addition he likes to thank Thomas Sauerwald for some clarifying discussions on the asymptotic interpretation of condition (LOG). Finally, the author likes to thank the referee for pointing out some mistakes in the submitted manuscript.

\setlength{\parskip}{0cm}

\end{document}